\documentclass[12pt]{article}
\usepackage{amsmath, amsfonts, amsthm, amssymb, color, hyperref, extarrows, enumitem, nicefrac,blindtext, mathtools,}

\newtheorem{theorem}{Theorem}[section]
\newtheorem{definition}[theorem]{Definition}

\newtheorem{lem}[theorem]{Lemma}
\newtheorem{pro}[theorem]{Proposition}

\numberwithin{equation}{section}

\newtheorem{remark}[theorem]{Remark}

\newtheorem{example}{Example}

\usepackage[hyperpageref]{backref}

\usepackage{cite}

\hypersetup{
	colorlinks   = true,
	citecolor    = magenta}
	
\begin{document}
\title{\vspace{-1cm} \bf Optimal $L^p$  regularity for $\bar\partial$ on the Hartogs triangle  \rm}
\author{ Yuan Zhang}
\date{}

\maketitle

\begin{abstract}
In this paper, we prove weighted  $L^p$ estimates for the canonical solutions on product domains. As an application, we show that if $p\in [4, \infty)$, the  $\bar\partial$ equation on the Hartogs triangle with $L^p$ data admits $L^p$ solutions with the desired estimates.  For any $\epsilon>0$, by constructing an example with $L^p$ data but having no $L^{p+\epsilon}$ solutions, we verify  the sharpness of the $L^p$ regularity on the Hartogs triangle. 

\end{abstract}

\renewcommand{\thefootnote}{\fnsymbol{footnote}}
\footnotetext{\hspace*{-7mm}
\begin{tabular}{@{}r@{}p{16.5cm}@{}}
& 2010 Mathematics Subject Classification. Primary 32W05; Secondary 32A25.\\
& Key words and phrases.
  $\bar{\partial}$ equation, Hartogs triangle, product domains, canonical solution, $L^p$ regularity,  Muckenhoupt's class.
\end{tabular}}

\section{Introduction}
Let $\mathbb{H}$ be the Hartogs triangle defined by $$ \mathbb{H} = \{(z_1, z_2)\in \mathbb{C}^2: |z_1|< |z_2| < 1\}. $$
Being a special bounded pseudoconvex domain without Lipschitz boundary,  the  Hartogs triangle has  played an important role in complex analysis and attracted considerable attention. See \cite{Shaw, CZ, BFLS, HW1, HW2, ChZ, Ch1, Ch2, CKY} et al. Among others a well-known result by Chaumat-Chollet \cite{ChC} in function theory   states  that  the  $\bar\partial$ problem  with smooth data on $\overline{\mathbb H}$ has no smooth solutions in general. 
On the other hand, when restricted at each H\"older level,  they  showed that the $\bar\partial$ equation   admits solutions in the same H\"older space as that of the data. (Note this does not contradict with the global irregularity, as the solution operators  are different at different H\"older levels.)

In view of a biholomorphism between the punctured bidisc and the Hartogs triangle,  a natural machinery was introduced  in works of Ma-Michel \cite{MM} and   Chakrabarti-Shaw \cite{CS2}   to treat with the  $\bar\partial$ problem on the Hartogs triangle. That is, using the biholomorphism  to pull back the data and  obtain a $\bar\partial$ equation on the punctured bidisc. Upon solving it via available integral representations on the punctured bidisc (or, on the bidisc after extension), use the biholomorphism again to push forward the solutions onto the Hartogs triangle. See also  a  recent joint work \cite{YZ} with Yuan   for some applications  to a general class of quotient domains. Since the push-forward and pull-back operators via the biholomorphism introduce a certain (nontrivial) weight, the regularity problem of $\bar\partial$ on the Hartogs triangle is reduced to a weighted $\bar\partial$ regularity problem  on the underlying bidisc. 

Motivated by these works and the machinery,  we study  the  weighted optimal $L^p$ estimates for the canonical solutions on general product domains, when the weights lie in a class of Muckenhoupt spaces $A^*_p$ (see Definition \ref{aps}).  Recall that the canonical solutions to $\bar\partial$ are the unique square-integrable solutions satisfying  the non-homogeneous Cauchy-Riemann equation $\bar\partial u = \mathbf f$ that are orthogonal to  $ker(\bar\partial)$.

\begin{theorem}\label{solution product}
Let $\Omega: =D_1\times \cdots\times D_n$ be a bounded product domain in $\mathbb C^n$,  where each $D_j$ has $C^2$ boundary, $j=1, \ldots, n$. Assume $\mu\in  A^*_p$, $ 1<p<\infty$. Then the canonical solution operator $T_c$ to $\bar\partial$ on $\Omega$ extends as a bounded operator from $L^{p}(\Omega, \mu)$ into itself.  Namely, there exists a constant $C$ dependent only on $\Omega$, $p$ and the $A^*_p$ constant of $\mu$ such that  for any $(0,1) $ forms $\mathbf f\in L^p(\Omega, \mu)$,
$$\|T_c \mathbf f\|_{L^p(\Omega, \mu)}\le C \|\mathbf f\|_{L^p(\Omega, \mu)}.$$
\end{theorem}

The main ingredient of the proof is  the weighted $L^p$ estimates for some Riesz-type integrals, as well as a pointwise estimate of the  canonical solution kernel established by Dong, Pan and the author \cite{DPZ} (see also an observation by Yuan \cite{Y}). According to an example of Kerzman (Example \ref{ex2}), the theorem gives  the optimal weighted $L^p$ regularity  on product domains in terms of the canonical solutions.  In particular, since $1\in A_p^*$ for all $p>1$, the canonical solutions  provides optimal solutions to $\bar\partial$ in the (unweighted) $L^p$ category as well (Example \ref{ex3}), unlike  another well-investigated solution operator along the line of Henkin which by a result of Chen-McNeal \cite{ChM1} is unbounded in $L^p, p<2$. 
\medskip

As an application of Theorem \ref{solution product}, we   obtain the following (unweighted) $L^p$ regularity for $\bar\partial$ on the Hartogs triangle if $p\ge 4$.

\begin{theorem}\label{main}
There exists a solution operator $T$ such that for any $\bar\partial$-closed $(0, 1)$ form $\mathbf f\in L^{p}(\mathbb H), 4\le p<\infty$,  $T\mathbf f\in L^p(\mathbb H)$ and solves  $\bar\partial u =\mathbf f$ on $\mathbb H$. Moreover, there exists a constant $C$ dependent only on   $p$  such that  for any $\bar\partial$-closed $(0,1)$ form $\mathbf f\in L^p(\mathbb H)$, \begin{equation*}
    \|T\mathbf f\|_{L^p(\mathbb H)}\le C\|\mathbf f\|_{L^p(\mathbb H)}.
\end{equation*}  
\end{theorem}

Unfortunately, our method only works in the special range $p\ge 4$. In fact, as shown in Lemma \ref{ex1}, this range   allows us to deal with a technical difficulty arisen from extending the $\bar\partial$-closed forms from the punctured bidisc to the whole bidisc, where Theorem \ref{solution product} can be applied. Remark \ref{66} further demonstrates that such a $\bar\partial$-closed extension to the bidisc  fails in general if $p<4$.   

  At the end of the paper, we construct an example (Example \ref{ex}) to demonstrate the sharpness of Theorem \ref{main}, in the sense that for any $ \epsilon>0$, there exists an $L^p$ datum which does not admit any $L^{p+\epsilon}$ solutions  on the Hartogs triangle. This non-improving phenomenon for the $\bar\partial$ regularity  on the Hartogs triangle  is essentially rooted from that for product domains. Theorem \ref{solution product} and the general framework  in \cite{YZ} can certainly be applied to  other special domains such as  proper holomorphic map images of product domains, which are left to interested readers. 

 \medskip
 
\noindent{\bf Acknowledgement:} The author thanks Professor Yifei Pan for helpful suggestions and Yuan Yuan for  discussions.

\section{Weighted $L^p$ estimates on product domains}
\subsection{Notations and Preliminaries}
We first introduce our weight space under consideration. Denote by $dV_z$ the Lebesgue integral element along the $z$ directions, and by  $|S|$  the Lebesgue measure of a subset  $S$ in $\mathbb C^n$.  For  $z=(z_1, \cdots, z_n)\in \mathbb C^n$, let $\hat z_j =(z_1, \cdots, z_{j-1}, z_{j+1}, \cdots, z_n)\in \mathbb C^{n-1}$, where the $j$-th component of $z$ is skipped. 

\begin{definition}\label{aps}
Given $1<p<\infty$, a weight $\mu: \mathbb C^n\rightarrow [0, \infty)$ is said to be in $ A^*_p$ if
$$ A_p^*(\mu): = \sup \left(\frac{1}{|D|}\int_{D}\mu(z)dV_{z_j}\right)\left(\frac{1}{|D|}\int_{D} \mu(z)^{\frac{1}{1-p}}dV_{z_j}\right)^{p-1}<\infty, $$
 where the supremum is taken over  almost every $\hat z_j\in \mathbb C^{n-1}, j=1, \ldots, n $, and   all discs $D\subset \mathbb C$. 
\end{definition}

We also recall  the standard  Muckenhoupt's class $A_p$, which consists of all weights $\mu: \mathbb C^n\rightarrow [0, \infty)$ satisfying 
 \begin{equation*}
   A_p(\mu): =  \sup \left(\frac{1}{|B|}\int_{B}\mu(z)dV_z\right)\left(\frac{1}{|B|}\int_{B} \mu(z)^{\frac{1}{1-p}}dV_z\right)^{p-1}<\infty, \end{equation*}
where the supremum is taken over all balls $B\subset \mathbb C^n$. See \cite[Chapter V]{Stein} for an introduction of the class. It is not hard to see that $A_q\subset A_p$ if $1< q<p$. Moreover, $A_p$ spaces satisfy an open-end property:  if $\mu\in A_p$ for some $1<p<\infty$, then  $\mu\in A_{\tilde p} $ for some ${\tilde p}<p$. 

Clearly,   $A^*_p = A_p$ when $n=1$. In general, $\mu\in A^*_p$ if and only if the $\delta$-dilation $\mu_\delta(z): =\mu(\delta_1z_1, \ldots, \delta_n z_n)\in A_p$  with a uniform $A_p$ constant for all $\delta =(\delta_1, \ldots, \delta_n)\in  (\mathbb R^+)^n$ (\cite[pp. 454]{GR}).   This in particular implies  $A^*_p\subset A_p$.    As will be seen in the rest of the paper,  the setting of $A^*_p$ weights allows us to apply the  slicing property  of product domains rather effectively.

\medskip

 Given a non-negative weight $\mu$ and $ 1< p<\infty $, the weighted function space $L^p(\Omega, \mu)$ is the set of  functions $f$ on $\Omega$ such that its weighted $L^p$ norm $$ \|f\|_{L^p(\Omega, \mu)}: = \left(\int_\Omega |f(z)|^p\mu(z)dV_z\right)^\frac{1}{p}<\infty. $$
When $\mu\equiv 1$, it is reduced to the (unweighted) $L^p(\Omega)$ space. From now on, we shall say $a\lesssim b$ if  $a\le Cb$ for a constant $C>0$ dependent only possibly on $\Omega, p$ and the $  A_p^* $ (or $A_p$) constant of $\mu$. We say  $a\approx b$ if and only if $a\lesssim b$ and $b\lesssim a$ at the same time.

\subsection{Weighted $L^p$ estimates for Riesz-type integrals}
We focus on a bounded product domain $\Omega = D_1\times \cdots \times D_n\subset\mathbb C^n$, where $D_j$ has $C^2$ boundary.  Fixing a multi-index $\alpha=(\alpha_1, \ldots, \alpha_n)$ with $0< \alpha_j<2, j=1, \ldots, n$ and $ 1<p<\infty$,   define the following Riesz-type integral of $f\in L^p(\Omega)$
$$ R_\alpha f(z): = \int_{\Omega} \frac{f(\zeta)}{\prod_{j=1}^n|\zeta_j-z_j|^{\alpha_j}}dV_\zeta, \ \ z\in \Omega.$$
 $R_\alpha$ is  a bounded operator from $L^p(\Omega)$ into itself by Riesz integral theory. We shall show a weighted version of this result in $L^p(\Omega, \mu)$ under the assumption that $\mu\in A_p^*$. 
 
 Firstly we consider  the $n=1$ case below, whose proof is a slight modification of  a standard trick for fractional integrals. 

\begin{pro}\label{r1}
Let $D$ be a bounded domain in $\mathbb C$. Assume $ 0<\alpha<2$ and $ \mu\in  A_p, 1<p<\infty$. Then $R_\alpha$ is a bounded operator from $L^p(D, \mu)  $ into itself. Namely, 
\begin{equation}\label{n1}
\|R_\alpha f\|_{L^p(D, \mu)}\lesssim \|f\|_{L^p(D, \mu)}.     
\end{equation}
\end{pro}

\begin{proof}
Without loss of generality, assume $f\ge 0$ and $f$ trivially extends to $\mathbb C$ by letting it be zero outside $D$. Denote by $Mf$ the Hardy-Littlewood maximal function of $f$. For each $z\in \mathbb C$ with $\delta>0$ to be chosen later, 
\begin{equation}\label{1}
\begin{split}
       \int_{|\zeta-z|<\delta, \zeta\in D}\frac{f(\zeta)}{|\zeta -z|^\alpha}dV_\zeta=&\sum_{k=1}^\infty \int_{\frac{\delta}{2^{k}}<|\zeta-z|<\frac{\delta}{2^{k-1}}}\frac{f(\zeta)}{|\zeta -z|^\alpha}dV_\zeta \le \sum_{k=1}^\infty \frac{2^{k\alpha}}{\delta^\alpha} \int_{|\zeta-z|<\frac{\delta}{2^{k-1}}}f(\zeta)dV_\zeta\\
       \lesssim & \sum_{k=1}^\infty 2^{-k(2-\alpha)}\delta^{2-\alpha} Mf(z)   \approx  \delta^{2-\alpha} Mf (z).
\end{split}
\end{equation}
 Due to  the open-end property of $A_p$, we can pick some $\tilde p\in \left(\frac{(2-\alpha)p}{2}, p\right)$  such that    $\mu\in A_{\tilde p}$. Then
\begin{equation}\label{2}
    \begin{split}
        \int_{|\zeta-z|>\delta, \zeta\in D}  \frac{f(\zeta)}{|\zeta -z|^\alpha}dV_\zeta \le &\left(\int_{|\zeta-z|>\delta, \zeta\in D}|f(\zeta)|^p \mu(\zeta)dV_\zeta\right)^{\frac{1}{p}}\left(\int_{|\zeta-z|>\delta, \zeta\in D}|\zeta -z|^{\frac{\alpha p}{1-p}}\mu(\zeta)^{\frac{1}{1-p}}dV_\zeta\right)^{\frac{p-1}{p}}\\
        \lesssim& \|f\|_{L^p(D, \mu)}\left(\int_{|\zeta-z|>\delta, \zeta\in D } |\zeta-z|^{\frac{\alpha p}{{\tilde p}-p}  }dV_\zeta \right)^\frac{p-{\tilde p}}{p}\left(\int_D \mu(\zeta)^{\frac{1}{1-{\tilde p}}}dV_\zeta\right)^\frac{{\tilde p}-1}{p}\\
      \lesssim  & \delta^{2-\alpha-\frac{2\tilde p}{p}}\|f\|_{L^p(D, \mu)}\left(\int_D \mu(\zeta)^{\frac{1}{1-{\tilde p}}}dV_\zeta\right)^\frac{{\tilde p}-1}{p}.
    \end{split}
\end{equation}
Combining \eqref{1} and \eqref{2}, we have
\begin{equation*}
R_\alpha f(z)\lesssim \delta^{2-\alpha} Mf(z)+ \delta^{2-\alpha-\frac{2\tilde p}{p}}\|f\|_{L^p(D, \mu)}\left(\int_D \mu(\zeta)^{\frac{1}{1-{\tilde p}}}dV_\zeta\right)^\frac{{\tilde p}-1}{p}.
    \end{equation*}

 Choosing $\delta =   \left(\frac{\|f\|_{L^p(D, \mu)}\left(\int_D \mu(\zeta)^{\frac{1}{1-{\tilde p}}}dV_\zeta\right)^\frac{{\tilde p}-1}{p} }{Mf}\right)^{\frac{p}{2\tilde p} } $ in the above, we further get
\begin{equation*}
R_\alpha f(z)\lesssim \|f\|^{\frac{(2-\alpha)p}{2\tilde p} }_{L^p(D, \mu)}\left(\int_D \mu(\zeta)^{\frac{1}{1-{\tilde p}}}dV_\zeta\right)^\frac{(2-\alpha)(\tilde p-1)}{2\tilde p} Mf(z)^{\frac{2\tilde p - (2-\alpha)p}{2\tilde p} }.
    \end{equation*}
Note that  $2\tilde p - (2-\alpha)p>0$ by the choice of $\tilde p$. Making use of the boundedness of the maximal function operator in $L^p(\mathbb C, \mu)$,
\begin{equation*}
    \begin{split}
      \|R_\alpha f\|_{_{L^{\frac{2p{\tilde p}}{2\tilde p -(2-\alpha)p}}(D, \mu)}}\lesssim & \|f\|^{\frac{(2-\alpha)p}{2\tilde p} }_{L^p(D, \mu)}\left(\int_D \mu(\zeta)^{\frac{1}{1-{\tilde p}}}dV_\zeta\right)^\frac{(2-\alpha)(\tilde p-1)}{2\tilde p} \|Mf\|_{L^p(\mathbb C, \mu)}^\frac{2\tilde p -(2-\alpha)p}{2{\tilde p}}\\
      \lesssim &  \|f\|_{L^p(D, \mu)}\left(\int_D \mu(\zeta)^{\frac{1}{1-{\tilde p}}}dV_\zeta\right)^\frac{(2-\alpha)(\tilde p-1)}{2\tilde p}.
    \end{split}
\end{equation*}
Lastly, since  $\mu\in A_{\tilde p}$, we have \begin{equation*}
    \begin{split}
        \|R_\alpha f\|_{_{L^p(D, \mu)}}^p\le &\|R_\alpha f\|^p_{_{L^{\frac{2p{\tilde p}}{2\tilde p -(2-\alpha)p}}(D, \mu)}} \left(\int_{D}\mu(\zeta)dV_\zeta\right)^\frac{(2-\alpha)p}{2\tilde p}\\
        \lesssim&  \|f\|^p_{L^p(D, \mu)}\left(\int_D \mu(\zeta)^{\frac{1}{1-{\tilde p}}}dV_\zeta\right)^\frac{(2-\alpha)(\tilde p-1)p}{2\tilde p}   \left(\int_{D}\mu(\zeta)dV_\zeta\right)^\frac{(2-\alpha)p}{2\tilde p}            \lesssim  \|f\|^p_{L^p(D, \mu)}.
    \end{split}
    \end{equation*}
\end{proof}

It is worth pointing out that $\mu\in A_p$ can not be dropped in Proposition \ref{r1}.  Indeed, letting $\triangle$ be the unit disc on $\mathbb C$, a function $f\in  L^2(\triangle, |z|^2)$ was constructed in \cite{YZ} such that     $R_1f\notin  L^2(\triangle, |z|^2)$.  Note that $|z|^2 \notin A_2$.

\begin{theorem}\label{rn}
Let $\Omega$ be a bounded product domain in $\mathbb C^n, n\ge 1$. Assume    $\alpha=(\alpha_1, \ldots, \alpha_n)$ with $0< \alpha_j<2, j=1, \ldots, n$,  and $\mu\in  A^*_p, 1<p<\infty$. Then $R_\alpha$ is a bounded operator from $L^p(\Omega, \mu)  $ into itself. Namely, 
\begin{equation*}
\|R_\alpha f\|_{L^p(D, \mu)}\lesssim \|f\|_{L^p(D, \mu)}.     
\end{equation*}
\end{theorem}

\begin{proof}
$n=1$ case is due to Proposition \ref{r1}. We shall prove $n\ge 2$ cases by induction. Denote by $\alpha'$ the first $n-1$ components of $\alpha$. Similarly define $z', \zeta'$ and $ \Omega'$. Write $ \|R_\alpha f\|^p_{L^p(\Omega, \mu)} $ as
\begin{equation*}
    \begin{split}
       \int_{D_1\times \cdots D_{n-1}}\int_{D_n}\left|\int_{D_n}\frac{1}{|\zeta_n-z_n|^{\alpha_n}}\left(\int_{D_1\times \cdots D_{n-1}}\frac{f(\zeta)}{|\zeta'-z'|^{\alpha'} } dV_{\zeta'}\right)dV_{\zeta_n}\right|^p\mu(z', z_n) dV_{z_n}dV_{z'}. 
    \end{split}
\end{equation*}
For almost everywhere fixed  $z'\in \Omega'$, note that  $\mu(z', \cdot)\in A_p$ by definition. Applying \eqref{n1}  to  $\int_{D_1\times \cdots D_{n-1}}\frac{f(\zeta)}{|\zeta'-z'|^{\alpha'} } dV_{\zeta'} $ on $D_n$,  we have \begin{equation*}
\begin{split}
       & \int_{D_n}\left|\int_{D_n}\frac{1}{|\zeta_n-z_n|^{\alpha_n}}\left(\int_{D_1\times \cdots D_{n-1}}\frac{f(\zeta)}{|\zeta'-z'|^{\alpha'} } dV_{\zeta'}\right)dV_{\zeta_n}\right|^p\mu(z', z_n) dV_{z_n}\\
    \lesssim &\int_{D_n}\left|\int_{D_1\times \cdots D_{n-1}}\frac{f(\zeta', z_n)}{|\zeta'-z'|^{\alpha'} } dV_{\zeta'}\right|^p\mu(z) dV_{z}.  
    \end{split}
\end{equation*} 
Thus 
\begin{equation*}
    \begin{split}
  \|R_\alpha f\|^p_{L^p(\Omega, \mu)} \lesssim   &  \int_{D_1\times \cdots D_{n-1}\times D_n}\left|\int_{D_1\times \cdots D_{n-1}}\frac{f(\zeta', z_n)}{|\zeta'-z'|^{\alpha'} } dV_{\zeta'}\right|^p\mu(z', z_n) dV_{z}\\
  =&  \int_{D_n}\left(\int_{D_1\times \cdots \times D_{n-1}}\left|\int_{D_1\times \cdots D_{n-1}}\frac{f(\zeta', z_n)}{|\zeta'-z'|^{\alpha'} } dV_{\zeta'}\right|^p\mu(z', z_n) dV_{z'}\right)dV_{z_n}\\
   \lesssim&\cdots\\
  \lesssim &  \int_{D_1\times \cdots D_{n-1}\times D_n}|f(z)|^p\mu(z) dV_{z} = \|f\|^p_{L^p(\Omega, \mu)},
    \end{split}
    \end{equation*}
where in the omitted part, we have employed a standard induction to the term inside the parenthesis  for almost everywhere fixed $z_n\in D_n$.
 
\end{proof}

\subsection{Proof of Theorem \ref{solution product}}
The (unweighted) $\bar\partial$ theory on product domains has been thoroughly understood through the works, for instance,  \cite{NW, henkin, CS1, ChM1, ChM2, FP, JY, PZ, Zh} and the references therein. In particular, it was proved in \cite{DPZ} that   there exists  a family of  functions $e_w$ on $\Omega$ such that 
\begin{equation}\label{ca}
     T_c\mathbf f (z) = \sum_{s=1}^n \sum_{1\le i_1<\cdots<i_s\le n} \sum_{k=1}^s
\int \limits_{D_{i_1}\times \cdots\times D_{i_s}} f_{i_k}(\zeta', z'')\frac{\partial^{s-1} e_w^{k, i_1, \ldots, i_s}(\zeta)}{\partial \bar{\zeta}_{i_1}\cdots\partial \bar \zeta_{i_{k-1}}\partial \bar \zeta_{i_{k+1}}\cdots\partial \bar{\zeta}_{i_s}}
\end{equation}
 is the canonical solution to $\bar\partial u = \mathbf f (=\sum_{j=1}^n f_jd\bar z_j)$ on $\Omega$. In fact,   after the preprint \cite{DPZ} was submitted to a journal, we observed   no boundary integrals should be involved in the solution representation originally constructed in \cite[Section 5]{DPZ}, due to the vanishing property of the kernels. This leads to the above simplified expression of the canonical solution.  Moreover, formula (5.2) in \cite{DPZ} with $s=m+1$ states that there exists some constant $0<\alpha_j<2$ such that $e_w$ satisfies \begin{equation} \label{es}
 \left|\frac{\partial^{s-1} e_w^{k, i_1,\ldots, i_s}}{\partial \bar{\zeta}_{i_1}\cdots\partial \bar \zeta_{i_{k-1}}\partial \bar \zeta_{i_{k+1}}\cdots\partial \bar{\zeta}_{i_s}}\right|\lesssim \frac{1} {\prod_{r=1}^s|{\zeta}_{i_r}-w_{i_r}|^{\alpha_j}}
\end{equation} 
 on $\Omega$. The unweighted $L^p$ theory for the canonical solution operator on product domains follows immediately from \eqref{es} by Young's inequality. Recently the same observation on the vanishing of the boundary integrals was made in \cite{Y} as well (we contacted the author right away and provided with our   email record in 2021 on the observation and the completed proof to the $L^p$  estimates).  
 
 It turns out our weighted $L^p$ estimate  Theorem  \ref{solution product}  is essentially a consequence of \eqref{es} and Theorem \ref{rn}.

\begin{proof}[Proof of Theorem \ref{solution product}:]
We estimate the term  in \eqref{ca} with $ (i_1, \ldots, i_s) = (1, \ldots, s) $ and $i_k=s$. Let $\alpha= (\alpha_1, \ldots, \alpha_s)$ satisfying \eqref{es}. Denote $z'=(z_1, \ldots, z_s)$ and $z'' = (z_{s+1},\ldots, z_n)$, and similarly define $D'$ and $D''$. By Fubini theorem and \eqref{es},
 \begin{equation*}
     \begin{split}
         &\left\|\int \limits_{D_{i_1}\times \cdots\times D_{i_s}} f_{i_k}(\zeta', z'')\frac{\partial^{s-1} e_w^{k, i_1, \ldots, i_s}(\zeta)}{\partial \bar{\zeta}_{i_1}\cdots\partial \bar \zeta_{i_{k-1}}\partial \bar \zeta_{i_{k+1}}\cdots\partial \bar{\zeta}_{i_s}}\right\|^p_{L^p(\Omega, \mu)} \\
         \lesssim  &\int_{D''}\int_{D'}|R_\alpha f_s(z', z'') |^p\mu(z', z'')dV_{z'} dV_{z''}\\
         \lesssim &\int_{D''}\int_{D'}| f_s(z', z'') |^p\mu(z', z'')dV_{z'} dV_{z''} \lesssim \|\mathbf f\|^p_{L^p(\Omega, \mu)}.
     \end{split}
 \end{equation*}
Here we used Theorem \ref{rn} on $D'$ in the second inequality  for almost every fixed $z''\in D''$. The rest of the terms in the sum in $T_c\mathbf f$ are proved similarly. The proof is complete. 

\end{proof}

The following example along the line of Kerzman \cite{Ker}  demonstrates that given a weighted $L^p$ data, the $\bar\partial$ problem on product domains  in general   does not  admit weighted $L^{p+\epsilon}, \epsilon>0$ solutions.  Thus Theorem \ref{solution product} gives  the optimal weighted $L^p$ regularity  on product domains in terms of the canonical solutions. 

\begin{example}\label{ex2}
  For each  $ 1<p<\infty, \epsilon>0$ and any $r\in \left(\frac{2}{1+\epsilon}, 2\right)$, consider $\mathbf f= (z_2-1)^{-r}d\bar z_1 $ on $\triangle^2$, $\frac{1}{2}\pi <\arg (z_2-1)<\frac{3}{2}\pi$ and $\mu =|z_2-1|^{r(p-1)}$. Then $\mu\in A_p^*$, $\mathbf f\in L^{p}(\triangle^2, \mu )$  and   is  $\bar\partial$-closed on $\triangle^2$. However, there does not exist a solution $u\in L^{p+\epsilon}(\triangle^2, \mu)$ to $\bar\partial u =\mathbf f$ on $\triangle^2$.  \end{example}

\begin{proof}Clearly $\mathbf f$ is $\bar\partial$-closed on $\triangle^2$. Since $r<2$, we can also verify directly that $\mu\in A_p^*$ and  $\mathbf f\in L^{p}(\triangle^2, \mu) $.   Suppose there exists some $u\in L^{p+\epsilon}(\triangle^2, \mu)$ satisfying $\bar\partial u =\mathbf f $ on $\triangle^2$. Noting that $L^{p}(\triangle^2, \mu)\subset L^1(\triangle^2)  $, an application of Weyl's lemma  gives the existence of some holomorphic function $h$ on $\triangle^2 $, such that $u = (z_2-1)^{-r}\bar z_1+h \in L^{p+\epsilon}(\triangle^2, \mu)$. 

For almost everywhere $(r, z_2) \in U: =  (0,1) \times \triangle\subset \mathbb R^3$, consider  
    $$v(r, z_2): =\int_{|z_1|= r} {u}(z_1, z_2) dz_1. $$
   By H\"older inequality, Fubini theorem and the fact that $p>1$,  \begin{equation*}
        \begin{split}
            \|v\|^{p+\epsilon}_{L^{p+\epsilon}(U, \mu)} =&\int_{U}  \left|\int_{|z_1|= r} {u}(z_1, z_2) dz_1\right|^{p+\epsilon}\mu(z_2)dV_{z_2, r}\\
            \le&  \int_{|z_2|<1}\int_{r<1}\left|r  \int_{0}^{2\pi} |{u}(re^{i\theta}, z_2 )| d\theta  \right|^{p+\epsilon} dr\mu(z_2)dV_{z_2} \\
            \lesssim & \int_{|z_2|<1}\int_{r<1} \int_{0}^{2\pi} |{u}(re^{i\theta}, z_2 )|^{p+\epsilon}d\theta r dr \mu(z_2) dV_{z_2} \\
            = & \int_{|z_2|<1, |z_1|<1 } |{u}(z )|^{p+\epsilon}\mu(z_2)dV_{z}= \|{u}\|^{p+\epsilon}_{L^{p+\epsilon}(\triangle^2, \mu)}<\infty.
        \end{split}
    \end{equation*} 
Thus  $v\in L^{p+\epsilon}(U, \mu)$.   On the other hand, by Cauchy's theorem, for almost everywhere $(r, z_2)\in U$,
  \begin{equation*}
     v(r, z_2) =\int_{|z_1|=r}  (z_2-1)^{-r}\bar z_1dz_1 =  (z_2-1)^{-r}\int_{|z_1|=r} \frac{r^2}{ z_1}dz_1 = 2\pi r^2i  (z_2-1)^{-r},
  \end{equation*}
  which is not in  $L^{p+\epsilon}(U, \mu)$ by the choice of $r>\frac{2}{1+\epsilon}$. This is a  contradiction! The example is thus verified. 
  
\end{proof}

In the case when $\mu\equiv 1$, one can  verify similarly  the following example to conclude that  the $\bar\partial$ problem on product domains  does not improve regularity in $L^p$ space, either. Thus the canonical solutions provide  optimal  $L^p$ solutions  on product domains. Example \ref{ex3} can easily be tailored to show that the $\bar\partial$ operator  does not improve regularity in unweighted $W^{k, p}$ spaces as well. This phenomenon is consistent with that in H\"older spaces (\cite{PZ, Zh}).

\begin{example}\label{ex3}
  For each  $ 1<p<\infty$, let $\mathbf f= (z_2-1)^{-\frac{2}{{p}}}d\bar z_1 $ on $\triangle^2$, $\frac{1}{2}\pi <\arg (z_2-1)<\frac{3}{2}\pi$. Then $\mathbf f\in L^{\tilde  p}(\triangle^2 )$ for all $1<\tilde  p< p$ and   is  $\bar\partial$-closed on $\triangle^2$. However, there does not exist a solution $u\in L^{p}(\triangle^2)$ to $\bar\partial u =\mathbf f$ on $\triangle^2$.  
\end{example}

\begin{proof} The proof is similar to that of Example \ref{ex2} with $\mu\equiv 1$ instead, so we only sketch it here. Clearly $\mathbf f\in L^{\tilde p}(\triangle^2 ) $  for all $1<\tilde  p< p$. Suppose there exists some $u\in L^p(\triangle^2 )$ satisfying $\bar\partial u =\mathbf f $ on $\triangle^2$.  Then for some holomorphic function $h$ on $\triangle^2 $ we have $u = (z_2-1)^{-\frac{2}{{p}}}\bar z_1+h \in L^p(\triangle^2)$. 
Consider 
    $v(r, z_2): =\int_{|z_1|= r} {u}(z_1, z_2) dz_1$ for almost everywhere $(r, z_2) \in U: =  (0,1) \times \triangle\subset \mathbb R^3$. 
  Then     $v\in L^p(U)$. However, by Cauchy's theorem, 
 $     v(r, z_2)  = 2\pi r^2i  (z_2-1)^{-\frac{2}{{p}}}$
 almost everywhere in $U$,   which contradicts with the fact that $v\in L^p(U)$. 
 \end{proof}


\section{$L^p$ estimates on the Hartogs triangle}
Denote by  $\triangle^*: =\triangle\setminus \{0\}$,   the punctured disc on $\mathbb C$. Then   $\psi:  \triangle\times \triangle^* \rightarrow \mathbb H$ given by $$(w_1, w_2)\mapsto (z_1, z_2)= \psi(w)= (w_1w_2, w_2)$$ is a biholomorphism, with its inverse $\phi:  \mathbb H \rightarrow \triangle\times \triangle^*$ given by $$(z_1, z_2)\mapsto (w_1, w_2) = \phi(z) = \left(\frac{z_1}{z_2}, z_2\right).$$ This biholomorphism allows us to pull back and push forward  between $\mathbb H$ and $\triangle\times\triangle^*$. Due to the explicit and simple form of   $\psi$, we shall be self-contained and chase concretely how the  singularity affects the $\bar\partial$-closedness, the pull-back of the data and push-forward of (solution) functions. The general framework can be founded in \cite{YZ}.
In fact,  for any $\mathbf f = \sum_{j=1}^2 f_j(z)d\bar z_j\in L^p(\mathbb H) $,  making use of change of variables formula we have the pull-back \begin{equation}\label{55}
    \psi^* \mathbf f = f_1\circ \psi \cdot \bar w_2d\bar w_1 + \left(f_1\circ \psi \cdot\bar w_1 +f_2\circ \psi \right) d\bar w_2\in L^p(\triangle^2, |w_2|^2) 
\end{equation} with
\begin{equation}\label{pull}
     \|\psi^*\mathbf f\|^p_{L^p(\triangle^2, |w_2|^2) }\lesssim \sum_{j=1}^2 \int_{\triangle^2} |f_j\circ \psi(w)|^p|w_2|^2 dV_w =\sum_{j=1}^2 \int_{\mathbb H} |f_j(z)|^pdV_z = \|\mathbf f\|^p_{L^p(\mathbb H) }. \end{equation}
The inverse $\phi$ is used to push forward any function   $\tilde u\in L^p(\triangle^2, |w_2|^2  )$ to be in $ L^p(\mathbb H)$ with
\begin{equation}\label{push}
    \| \tilde u\circ \phi \|^p_{ L^p(\mathbb H)}= \int_{\mathbb H}  |\tilde u\circ \phi(z)|^pdV_z  = \int_{\triangle^2} |\tilde u(w)|^p|w_2|^2dV_w =\| \tilde  u\|^p_{L^{p}(\triangle^2, |w_2|^2) }.
\end{equation}

\medskip

Note that  $  |w_2|^2\in A^*_p$, $p>2$. In order to apply the weighted $L^p$ estimates in Theorem \ref{solution product} (where each portion $D_j$ are assumed to have $C^2$ boundary), we need to justify that the pull-back data is   $\bar\partial$-closed on $\triangle^2$.

\begin{pro}\label{44}
Let $\mathbf f\in L^p(\mathbb H)$ be a $\bar\partial$-closed $(0,1)$ form on $\mathbb H$. If $4\le p<\infty$, then $\psi^*\mathbf f  $ lies in $L^p(\triangle^2, |w_2|^2  )  $ and is  a $\bar\partial$-closed $(0,1)$ form on $\triangle^2$.
\end{pro}

The proof of  the proposition boils down to   showing the following Harvey-Polking type extension (or resolution) of $\bar\partial$-closed $ L^p(\triangle^2, |w_2|^2)$  forms  from $\triangle\times \triangle^*$ to $\triangle^2$,   $p\ge 4$. We remark that if the forms lie in $ W^{1,p}(\triangle, |w_2|^2) $ in addition, then this range of $p$ can be relaxed to $p>2$. See, for instance, \cite[Proposition 5.10]{YZ}.

\begin{lem}\label{ex1}
Suppose a $(0, 1)$ form $\mathbf h\in L^p(\triangle^2, |w_2|^2)$ is $\bar\partial$-closed  on $\triangle\times \triangle^* $. If $4\le p<\infty$, then $\mathbf h$  is $\bar\partial$-closed on $\triangle^2$.
\end{lem}

\begin{proof}
Write $\mathbf h(w) = h_1(w)d\bar w_1 +h_2(w)d\bar w_2$. Let $\eta = \eta(w) dw_1\wedge dw_2$ be a smooth $(2, 0)$-form in $\triangle^2 $ with compact support. We shall show 
$$-\left\langle \bar\partial h,  \eta \right\rangle_{\triangle^2}: =\int_{\triangle^2}  h_1(w)\frac{\partial \eta(w)}{\partial \bar w_2} - h_2(w)\frac{\partial \eta(w)}{\partial \bar w_1}   dV_w =0. $$
 
 Denote by $\triangle_r$ the disc centered at $0$ with radius $r>0$. Choose a cut-off function $\chi\in C_c^\infty(\triangle)$ such that $\chi =1$ on $\triangle_{\frac{1}{2}}$ and $ |\nabla \chi|<\frac{1}{3}$ on $\triangle$. Letting $\chi_k(w_2) =\chi(kw_2)$ on $\triangle$, then $\chi_k$ is supported on $ \triangle_{\frac{1}{k}}$ and $|\nabla \chi_k|\lesssim k$. Consequently,
\begin{equation*}
    \begin{split}
    & \left|\int_{\triangle^2} h_2(w)\frac{\partial \eta(w)}{\partial \bar w_1}  dV_w - \int_{\triangle^2}  h_1(w)\frac{\partial \left((1-\chi_k(w_2))\eta(w)\right)}{\partial \bar w_2} dV_w\right|\\
    \le & \int_{\triangle^2}  \left|h_2(w)\frac{\partial \left(\chi_k(w_2)\eta(w)\right)}{\partial \bar w_1}\right| dV_w  \\
    &+ \left| \int_{\triangle^2}  h_2(w)\frac{\partial \left((1-\chi_k(w_2))\eta(w)\right)}{\partial \bar w_1} -  h_1(w)\frac{\partial \left((1-\chi_k(w_2))\eta(w)\right)}{\partial \bar w_2}dV_w\right|.     \end{split}
\end{equation*}
Since $(1-\chi_k(w_2))\eta(w)$ has compact support on $\triangle\times \triangle^*$, the last line in the above is zero by the $\bar\partial$-closedness of $\mathbf h$ on $\triangle\times \triangle^*$. For the first term, since $p>2$,
\begin{equation}\label{hh}
    \begin{split}
    \int_{\triangle^2}  \left|h_2(w)\frac{\partial \left(\chi_k(w_2)\eta(w)\right)}{\partial \bar w_1}\right| dV_w \lesssim &   \int_{\triangle\times \triangle_{\frac{1}{k}} } | h_2(w)||w_2|^{\frac{2}{p}} |w_2|^{-\frac{2}{p}}| dV_w  \\
     \lesssim & \|h_2\|_{ L^p(\triangle^2, |w_2|^2)}\left(\int_{\triangle_{\frac{1}{k}} } |w_2|^{-\frac{2}{p -1}}dV_{w_2}\right)^{\frac{ p-1}{p}} \rightarrow 0
    \end{split}
\end{equation}
as $k\rightarrow \infty$. Hence as $k\rightarrow \infty$,
\begin{equation}\label{11}
   \left|\int_{\triangle^2} h_2(w)\frac{\partial \eta(w)}{\partial \bar w_1}  dV_w - \int_{\triangle^2}  h_1(w)\frac{\partial \left((1-\chi_k(w_2))\eta(w)\right)}{\partial \bar w_2} dV_w\right|\rightarrow 0. 
\end{equation}

On the other hand, 
\begin{equation*}
    \begin{split}
        \left| \int_{\triangle^2}  h_1(w)\frac{\partial \left(\chi_k(w_2)\eta(w)\right)}{\partial \bar w_2} dV_w  \right|         \lesssim \int_{\triangle\times \triangle_{\frac{1}{k}}}  \left|  h_1(w) \frac{\partial \left(\chi_k(w_2)\right)}{\partial \bar w_2} \right|dV_w   +    \int_{\triangle\times \triangle_{\frac{1}{k}}}  |h_1(w)|  dV_w.
    \end{split}
\end{equation*}
By the same reasoning as  in \eqref{hh}, the last term goes to $0$ as $k\rightarrow \infty$. For the first term in the right hand side of the last line, making use of the fact that $\left| \frac{\partial\left(\chi_k(w_2)\right)}{\partial \bar w_2}\right|\lesssim k $ on $ \triangle_{\frac{1}{k}}$, we get 
\begin{equation*}
    \begin{split}
        \left| \int_{\triangle\times \triangle_{\frac{1}{k}}}  h_1(w) \frac{\partial \left(\chi_k(w_2)\right)}{\partial \bar w_2} dV_w  \right|\lesssim & k \int_{\triangle\times \triangle_{\frac{1}{k}}}|h_1(w)||w_2|^{\frac{2}{p}} |w_2|^{-\frac{2}{p}}|dV_w\\
        \lesssim & k \|h_1\|_{ L^p(\triangle\times \triangle_{\frac{1}{k}}, |w_2|^2)}\left(\int_{\triangle_{\frac{1}{k}} } |w_2|^{-\frac{2}{p -1}}dV_{w_2}\right)^{\frac{ p-1}{p}} \\
        \lesssim& k^{\frac{4}{p}-1}\|h_1\|_{ L^p(\triangle\times \triangle_{\frac{1}{k}}, |w_2|^2)}.
    \end{split}
\end{equation*}
Since $p\ge 4$, as $k\rightarrow \infty$  the last term goes to zero, and thus
\begin{equation}\label{22}
    \left| \int_{\triangle^2}  h_1(w)\frac{\partial \left(\chi_k(w_2)\eta(w)\right)}{\partial \bar w_2} dV_w  \right|    \rightarrow 0. 
\end{equation}The proof of the proposition is complete by combining \eqref{11} and \eqref{22}. 

\end{proof}

\begin{remark}\label{66}
The  $p\ge 4$ assumption in Lemma \ref{ex1} can not be relaxed. For instance,  $\mathbf h(w) = \frac{1}{w_2}d\bar w_1$ is $\bar\partial$-closed  on $\triangle\times \triangle^* $ and lies in $L^p(\triangle^2, |w_2|^2  )$ for all  $1<p<4$.  However,  $\mathbf h(w)$ is not $\bar\partial$-closed  on $\triangle^2 $. In fact, given  a smooth $(2, 0)$-form $\eta = \eta(w) dw_1\wedge dw_2$ in $\triangle^2 $ with compact support,   by Stokes' theorem and Residue theorem, 
\begin{equation*}\begin{split}
 - \langle \bar\partial \mathbf h, \eta\rangle_{\triangle^2}  = &  \int_{\triangle^2}\frac{1}{w_2}\frac{\partial \eta(w)}{\partial \bar w_2} dV_w = \int_\triangle\lim_{\epsilon\rightarrow 0}  \int_{\triangle\setminus \triangle_\epsilon}\frac{1}{w_2}\frac{\partial \eta(w)}{\partial \bar w_2} dV_{w_2}dV_{w_1} \\
    = &\frac{1}{2}\int_\triangle\lim_{\epsilon\rightarrow 0}  \int_{\triangle\setminus \triangle_\epsilon}\bar\partial\left( \frac{ \eta(w)}{w_2}dw_2\right)dV_{w_1}  = -\frac{1}{2}\int_\triangle\lim_{\epsilon\rightarrow 0}  \int_{\partial \triangle_\epsilon} \frac{ \eta(w)}{w_2} dw_2dV_{w_1}\\
    =& -\pi i\int_\triangle\eta(w_1, 0)dV_{w_1},
     \end{split}
   \end{equation*}
   which is not zero in general. 
\end{remark} 

\medskip

\begin{proof}[Proof of Proposition \ref{44}: ]
In view of Lemma \ref{ex1} and \eqref{pull}, we only need to verify that    $\psi^*\mathbf f $ is $\bar\partial$-closed on $ \triangle\times \triangle^*$. Indeed,   for any smooth function $\chi$ with compact support on $\mathbb H$, by   $\bar\partial$-closedness of $\mathbf f$  on $\mathbb H$, we have 
\begin{equation}\label{33}
    \begin{split}
      \int_{\mathbb H} f_1(z) \frac{\partial \chi}{\partial \bar z_2}(z) -f_2(z)\frac{\partial \chi}{\partial \bar z_1}(z) dV_z =0.
    \end{split}
\end{equation}
For any $(2,0)$ form $\eta$ with compact support on $\triangle\times \triangle^*,$    by \eqref{55}, chain rule and change of variables formula,
\begin{equation*}
    \begin{split}
       - \langle \bar\partial\psi^*\mathbf f, \eta\rangle_{\triangle^2} = & \int_{\triangle^2}  f_1\circ \psi(w)\bar w_2 \frac{\partial \eta(w)}{\partial \bar w_2} - \left(f_1\circ \psi(w) \bar w_1 +f_2\circ \psi(w)\right)\frac{\partial \eta(w)}{\partial \bar w_1}   dV_w\\
        =&\int_{\mathbb H} \left(  f_1(z)\bar z_2\left(\frac{\partial \eta\circ \phi(z)}{\partial \bar z_1}\frac{\bar z_1}{\bar z_2}+ \frac{\partial \eta\circ \phi(z)}{\partial \bar z_2} \right)  - \left(f_1(z) \frac{\bar z_1}{\bar z_2 } +f_2(z)\right)\frac{\partial \eta\circ \phi(z)}{\partial \bar z_1} \bar z_2\right) \frac{1}{|z_2|^2} dV_z \\
        =& \int_{\mathbb H}   f_1(z) \frac{\partial }{\partial \bar z_2} \left(\frac{\eta\circ \phi(z)}{z_2}\right)   -  f_2(z) \frac{\partial }{\partial \bar z_1}  \left( \frac{\eta\circ \phi(z)}{z_2}\right) dV_z =0,
    \end{split}
\end{equation*}
where the last line is due to \eqref{33} and the fact that $ \frac{\eta\circ \phi(z)}{z_2} $ is smooth with compact support on $\mathbb H$. 

\end{proof}

\begin{proof}[Proof of Theorem \ref{main}:] Given a $\bar\partial$-closed $(0, 1)$ form $\mathbf f\in L^p(\mathbb H)$,  $\psi^*\mathbf f\in L^p(\triangle^2, |w_2|^2)$  and  is $\bar\partial$-closed on $\triangle^2$ by Proposition \ref{44}. As $|w_2|^2\in A_p^*$, an application of Theorem \ref{solution product} gives a solution $\tilde u\in L^p(\triangle^2, |w_2|^2)$ to $\bar\partial \tilde u = \psi^*\mathbf f$ on $\triangle^2 $.  Namely, for any smooth and compactly supported $(2,1)$ form $\eta = (\eta_1 d\bar w_1 +\eta_2 d\bar w_2)\wedge dw_1\wedge dw_2 $ on $\triangle^2$,
\begin{equation}\label{221}
\begin{split}
   &\int_{\triangle^2}  \tilde u(w)\left(   \frac{\partial\eta_1(w)}{\partial \bar w_2} -\frac{\partial\eta_2(w)}{\partial \bar w_1} \right)dV_w  =  -\langle \bar\partial \tilde u, \eta  \rangle_{\triangle^2} = -\langle \psi^* \mathbf f, \eta  \rangle_{\triangle^2}  \\
    =& \int_{\triangle^2} f_1\circ \psi(w) \bar w_2\eta_2(w) - \left(f_1\circ \psi(w) \bar w_1 +f_2\circ \psi(w)\right) \eta_1(w) dV_w.
    \end{split}
\end{equation}

We now verify that  $u: =\tilde u\circ \phi$ solves $\bar\partial u =\mathbf f$ on $\mathbb H$. Indeed, for any smooth $(2, 1)$ form $\chi = (\chi_1 d\bar z_1 +\chi_2 d\bar z_2)\wedge dz_1\wedge dz_2$ with compact support on $\mathbb H$,     by chain rule and change of variables,
\begin{equation*}
\begin{split}
       - \langle \bar\partial u, \chi\rangle_{\mathbb H} =& \int_{\mathbb H} \tilde u\circ \phi(z)\left(   \frac{\partial\chi_1(z)}{\partial \bar z_2} -\frac{\partial\chi_2(z)}{\partial \bar z_1} \right)dV_z   \\
    =& \int_{\triangle^2} \tilde u(w)\left(- \frac{\partial\chi_1\circ \psi(w)}{\partial \bar w_1}\frac{\bar w_1}{\bar w_2} + \frac{\partial\chi_1\circ \psi(w)}{\partial \bar w_2}- \frac{\partial\chi_2\circ \psi(w)}{\partial \bar w_1} \frac{1}{\bar w_2}  \right)|w_2|^2dV_w   \\
    =& \int_{\triangle^2} \tilde u(w)\left(- \frac{\partial\chi_1\circ \psi(w)}{\partial \bar w_1} \bar w_1 w_2 + \frac{\partial\chi_1\circ \psi(w)}{\partial \bar w_2} |w_2|^2- \frac{\partial\chi_2\circ \psi(w)}{\partial \bar w_1}  w_2  \right)dV_w \\
    =&\int_{\triangle^2} \tilde u(w)\left( \frac{\partial  }{\partial \bar w_2 }\left(  \chi_1\circ\psi(w)|w_2|^2\right)     - \frac{\partial  }{\partial \bar w_1 }\left( \chi_2\circ \psi(w)  w_2+\chi_1\circ\psi(w)  \bar w_1  w_2  \right)\right)dV_w. 
    \end{split}
\end{equation*} 
Note that $\eta_1(w): = \chi_1\circ\psi(w)|w_2|^2, \eta_2(w): =   \chi_2\circ \psi(w)  w_2+\chi_1\circ\psi(w)  \bar w_1  w_2    $  are both smooth with compact supports on $\triangle\times \triangle^*$ and thus on $\triangle^2$. Making use of \eqref{221}, we further simplify the above to get
\begin{equation*}
    \begin{split}
   -\langle \bar\partial u, \chi\rangle_{\mathbb H}   =& \int_{\triangle^2} \left(f_1\circ \psi(w)\chi_2\circ \psi(w) - f_2\circ \psi(w)\chi_1\circ \psi(w)\right)|w_2|^2 dV_w\\
    =&\int_{\mathbb H} f_1(z)\chi_2(z) - f_2(z)\chi_1(z) dV_z =-\langle \mathbf f, \chi\rangle_{\mathbb H}. 
    \end{split}
\end{equation*} 

Altogether, by \eqref{pull}-\eqref{push} and Theorem \ref{solution product},  the solution $u =\tilde u\circ \phi$ satisfies $$\|u\|_{ L^p(\mathbb H)} = \|\tilde  u\|_{L^{p}(\triangle^2, |w_2|^2) } \lesssim \| \psi^*\mathbf f \|_{ L^{p}(\triangle^2, |w_2|^2)}\lesssim \|\mathbf f\|_{L^p(\mathbb H) }. $$
  The proof is complete. 
  
\end{proof}

The following example shows Theorem \ref{main} is optimal, in the sense that solutions can not lie in a better Lebesgue space than that of the data in general. 

\begin{example}\label{ex}
 For each  $ 1<p<\infty$, let $\mathbf f= z_2(z_2-1)^{-\frac{2}{{p}}}\left(\frac{1}{\bar z_2}d\bar z_1 - \frac{\bar z_1}{\bar z^2_2}d\bar z_2\right) $ on $\mathbb H$, $\frac{1}{2}\pi <\arg (z_2-1)<\frac{3}{2}\pi$. Then $\mathbf f\in L^{\tilde  p}(\mathbb H)$ for all $1<\tilde  p< p$ and   is  $\bar\partial$-closed on $\mathbb H$. However, there does not exist a solution $u\in L^{p}(\mathbb H)$ to $\bar\partial u =\mathbf f$ on $\mathbb H$.  
 \end{example}

 \begin{proof}
 Since $|z_1|<|z_2|$ on $\mathbb H$, we have $\mathbf f\in L^{\tilde  p}(\mathbb H) $ for all $1<\tilde  p< p$. A direct computation shows that $\mathbf f$ is $\bar\partial$-closed on $\mathbb H$. Moreover, $ \psi^* \mathbf f= w_2(w_2-1)^{-\frac{2}{{p}}}d\bar w_1 $ and is $\bar\partial$-closed  on $\triangle^2$. 

Arguing by contradiction, suppose there exists some $u\in L^p(\mathbb H )$ satisfying $\bar\partial u =\mathbf f $ on $\mathbb H$. Then $\tilde u: = u\circ \psi $ lies in $ L^p(\triangle^2, |w_2|^2)$  and satisfies $\bar\partial \tilde u = \psi^*\mathbf f = \bar\partial \left(w_2(w_2-1)^{-\frac{2}{{p}}}\bar w_1\right)$ on $\triangle\times\triangle^*$. In particular,   $\tilde u\in L^p(\triangle\times (\triangle\setminus \triangle_{\frac{1}{2}}))$ and there exists some holomorphic function $h$ on $\triangle\times(\triangle \setminus \triangle_{\frac{1}{2}})$ such that $\tilde u = w_2(w_2-1)^{-\frac{2}{{p}}}\bar w_1+h \in L^p(\triangle\times (\triangle\setminus \triangle_{\frac{1}{2}}))$. 

For almost everywhere fixed $(r, w_2) \in U: =  (0,1)\times ( \triangle\setminus  \triangle_{\frac{1}{2}})\subset \mathbb R^3$, consider  
    $$v(r, w_2): =\int_{|w_1|= r} {\tilde u}(w_1, w_2) dw_1. $$
Similarly as in the proof of Example \ref{ex2} (with $\mu=1$), we see that $v\in L^p(U)$. 
Note that   $h(\cdot, w_2)$ is holomorphic on $\triangle$ for each fixed $w_2\in \triangle\setminus \triangle_\frac{1}{2}$. Thus  for almost everywhere fixed $(r, w_2)\in U$,  Cauchy's theorem gives
  \begin{equation*}
     v(r, w_2) =\int_{|w_1|=r} w_2(w_2-1)^{-\frac{2}{{p}}}\bar w_1dw_1  = 2\pi r^2i w_2(w_2-1)^{-\frac{2}{{p}}},
  \end{equation*}
  which does not belong to  $L^p(U)$. Contradiction!
  
 \end{proof}
 

\bibliographystyle{alphaspecial}

\fontsize{11}{11}\selectfont

\vspace{0.7cm}

\noindent zhangyu@pfw.edu,

\vspace{0.2 cm}

\noindent Department of Mathematical Sciences, Purdue University Fort Wayne, Fort Wayne, IN 46805-1499, USA.
\end{document}